\documentclass[11pt]{article}
\usepackage{amsfonts}
\usepackage{amssymb}
\usepackage{amsthm}
\usepackage{amsmath}
\usepackage{graphicx}
\usepackage{empheq}
\usepackage{indentfirst}
\usepackage{cite}
\usepackage{mathrsfs}
\usepackage{graphics}
\usepackage{enumerate}
\usepackage{color}

 \newtheorem{theorem}{Theorem}[section]
 
 \newtheorem{lemma}{Lemma}[section]

 \newtheorem{remark}{Remark}[section]
 \numberwithin{equation}{section}

\allowdisplaybreaks[4]

\newcommand{\nn}{\nonumber}

\def\e{\varepsilon}

\newcommand{\beq}{\begin{equation}}
\newcommand{\eeq}{\end{equation}}
 \def\non{\nonumber }
\def\bea{\begin{eqnarray}}
\def\eea{\end{eqnarray}}

\begin{document}
\title{Global Existence of Weak Solutions to a Signal-dependent Keller-Segel Model for Local Sensing Chemotaxis}
\author{Haixia Li\thanks{Innovation Academy for Precision Measurement Science and Technology, CAS,
Wuhan 430071, HuBei Province, P.R. China},
	\ Jie Jiang\thanks{Innovation Academy for Precision Measurement Science and Technology, CAS,
		Wuhan 430071, HuBei Province, P.R. China,
		\textsl{jiang@wipm.ac.cn, jiang@apm.ac.cn}.}}

\date{\today}

\maketitle

\begin{abstract}This paper is devoted to global existence of weak solutions to the following degenerate kinetic model of chemotaxis
	\begin{equation}
	\begin{cases}\label{chemo0}
	u_t=\Delta (\gamma (v)u)\\
	\tau v_{t}=\Delta v-v+u
	\end{cases}
	\end{equation}in a smooth bounded  domain with no-flux boundary conditions. The problem features a positive signal-dependent  motility function $\gamma(\cdot)$ which may vanish as $v$ becomes unbounded. In this paper, we first  modify the comparison approach developed recently in \cite{GM,GM2} to derive the upper bounds of $v$ under weakened assumptions on $\gamma(\cdot)$. Then by introducing a suitable approximation scheme which is compatible with the comparison method, we establish the global existence of weak solutions in any spatial dimension via compactness argument. Our weak solution has higher regularity than those obtained in previous literature \cite{Faguo,DKTY2019,Tao}.

\end{abstract}

\section{Introduction}
In this paper,  we consider the initial-Neumann boundary value problem of the following chemotaxis system
\begin{equation}
\begin{cases}
u_{t}=\Delta (\gamma(v)u), & x\in\Omega,\;t>0 \\
\tau v_{t}=\Delta v-v+u    ,  & x\in\Omega,\;t>0 \\
\frac{\partial u}{\partial \nu}=\frac{\partial v}{\partial\nu}=0, & x\in \partial\Omega,\;t>0\\
u(x,0)=u_0(x), \tau v(x,0)=\tau v_0(x),& x\in\Omega\label{1-3}
\end{cases}
\end{equation}
where $ \Omega\subset R^{n}$ with $n\geq 1$ is a smooth bounded domain. Here, $u$ and $v$ represents the density of cells and the concentration of chemical signals, respectively. $\tau\geq0$ is a given constant. This system features a signal-dependent motility function $\gamma(v)$, which is positive and may vanish as $v$ tends to infinity.

This model has been recently adopted in \cite{Sciencs11,PRL12} to describe the formation of stripe patterns via the so-called "self-trapping" mechanism, where $\gamma$ is a positive and decreasing function.  Thus the macroscopic cellular motility is suppressed by the concentration of signals. On the other hand, this model is a special version of the following Keller--Segel  model of chemotaxis originally proposed by Keller \& Segel in  the seminal work \cite{KELLER}: 
\begin{equation}
\begin{cases}\label{1-1}
u_{t} = \nabla\cdot(D(v)\nabla u-\chi(v)u\nabla v)\\
v_t=\Delta v-v+u.
\end{cases}
\end{equation}
Here, $D(\cdot)$ and $\chi(\cdot)$ denote the signal-dependent diffusivity and chemo-sensitivity, respectively, which  are  linked through
\begin{equation}
\chi(v)=(\alpha-1)D'(v).\non
\end{equation}
The coefficient $\alpha$ here is a parameter that is related to the distance between the signal-receptors in a cell with a suitable scaling. 
When $\alpha>0$, movement of a cell occurs in response to the transported  signals via the so-called "gradient sensing" mechanism since it  can perceive the gradient of concentrations by comparing them at two different spots. If $\alpha=0$, there is a single receptor in the cell and hence the cell can only detect the  concentration at  one spot. Note that a direct expansion of the right hand side of the first equation in \eqref{1-3} corresponds to the above Keller--Segel model with $\alpha=0$ and $D=\gamma$. Therefore, our system \eqref{1-3}  models the chemotaxis movement due to the above mentioned "local sensing" mechanism.

 The mathematical analysis of problem \eqref{1-3}  has been carried out recently in several works. By presuming strictly positive upper and lower bounds on $\gamma$ and $|\gamma'|$,  Tao \& Winkler \cite{Tao}  investigated the problem when $\tau=1$, where  existence of global classical solutions in the two dimensions and global weak solutions in higher dimensions was proved. Note degeneracy was precluded due to their assumptions.
 
If the motility function vanishes as $v$ becomes unbounded, then degeneracy brings a severe difficulty in analysis. With a specified power type decreasing and asymptotically vanishing motility  $\gamma(v)=c_{0}v^{-k}$ with $c_{0}>0,k>0$, Yoon \& Kim \cite{Yoon} obtained the global existence of classical solutions with uniform-in-time bounds when $\tau=1$ and  $c_{0}$ is sufficiently small. If $\tau=0$ , Ahn et al \cite{Ahn} removed the above smallness assumption and  established the global existence of classical solutions with uniform-in-time bounds when $n\leq2$ for any $k>0$ or $n\geq3$ for $k<\frac{2}{n-2}$. When $\tau=1$ and $\gamma(v)=\frac{1}{c+v^{k}}$ with some $c\geq0,k>0,$ the global weak solution was obtained for all $k>0$ if $n=1$, for $0<k<2$ if $n=2$, and for $0<k<\frac{4}{3}$ if $n=3$  in \cite{DKTY2019}. There are also many study concerning with global existence of system \eqref{1-3} with the presence of logistic source terms by replacing the first equation by
\begin{equation*}
	u_t=\Delta(\gamma(v)u)+\mu u(1-u)
\end{equation*}
with some $\mu>0$. We refer the readers to \cite{Jin,GM2,LX2019,WW} for more detail.
 
In the work mentioned above, global classical solution was proved basically via conventional energy method. In fact, in order to prevent degeneracy one needs to deduce an upper bound estimate for $v$. To this aim, an indirect way adopted there was to establish the $L^\infty_tL^p_x-$boundedness of $u$, which will give rise to boundedness of $v$ according to the second equation by classical regularity theories. However, this idea seems only work for above mentioned special cases. 

Recently, based on a careful observation of the delicate structure of system \eqref{1-3}, Fujie \& Jiang \cite{GM,GM2}  proposed a new comparison method such that an explicit point-wise upper bound estimate of $v$ was derived for a very generic motility function, that is
\begin{equation}\label{gamma0a}
\gamma(\cdot)\in C^3[0,+\infty),\;\gamma(\cdot)>0,\;\;\gamma'(\cdot)\leq0\;\;\text{on}\;(0,+\infty)
\end{equation}
and additionally if $\tau>0$, we need  the following asymptotically vanishing property:
\begin{equation}\label{gamma}
\lim\limits_{s\rightarrow+\infty}\gamma(s)=0.
\end{equation}
 The key idea was to introduce  an auxiliary elliptic problem that enjoys the comparison principle. Then global existence of classical solution with a generic motility was proved in two dimensions and uniform-in-time boundedness was further discussed when $\gamma$ satisfied certain decay rate assumptions; see also \cite{FJ3}. Moreover, a new  critical mass phenomenon  was discovered in 2D when $\gamma(v)=e^{-v}$ thanks to its energy-dissipation structure. It was proved that in the sub-critical case, the global solution is uniformly-in-time bounded while blowup is verified to take place in the super-critical case only at time infinity. Such a delayed blowup behavior of local sensing chemotaxis is distinct from the finite-time blowup phenomenon due to gradient sensing mechanism (see also \cite{Faguo}).  We mention that in the case $\gamma(v)=e^{-v}$,  global boundedness with sub-critical mass and possible blowup at an undetermined  time (finite or infinite) in the super-critical case was also proved in \cite{JW20}. Shortly afterwards, Burger et al \cite{Faguo} also verified that blowup must occur at time infinity by duality method and moreover, weak solution was obtained in any dimension when $\gamma(v)=e^{-v}$. Their method relied on an introduction of an adequate approximation procedure that conserves an dissipative energy and the duality structure.

In the present contribution, we aim to establish global existence of weak solutions for a generic motility function in any dimension. For simplicity, we assume that
\begin{equation}
(u_0,v_0)\in L^\infty(\Omega)\times{L^{\infty}\cap H^1(\Omega)},\quad u_0\geq0,\; v_0\geq0. \label{1-4}
\end{equation}
As to the motility, we  require that
\begin{equation}
	\gamma(\cdot)\in C^{3}[0,\infty), \; 0<\gamma(\cdot)\leq K_\gamma \quad \text{on} \;(0,+\infty) \label{1-5}
\end{equation}with some constant $K_\gamma>0.$ Note we remove the decreasing and the asymptotically vanishing properties on $\gamma$ in \eqref{gamma0a} and \eqref{gamma} but assume additionally it is bounded from above, which is certainly satisfied in the case $\gamma'\leq0.$

The main result of the present work on global existence of weak solutions is stated as follows.
\begin{theorem}\label{1.1}
 Suppose that $(u_{0},v_{0})$ satisfies \eqref{1-4}. Assume that $\gamma$ satisfies \eqref{1-5} and  $0\leq \tau<1/K_\gamma$. Let $T>0$. Then problem \eqref{1-3} possesses at least one non-negative global weak solution $(u,v)$ satisfying
 \begin{equation}
 \begin{cases}\label{regularity}
 u\in L^\infty(0,T;L^1(\Omega)\cap H^{-1}(\Omega))\cap L^{2}(0,T;L^{2}(\Omega)) \cap  L^{\frac{4}{3}}(0,T;W^{1,\frac{4}{3}}(\Omega)),\\
 v\in L^{\infty}(0,T;L^{\infty}(\Omega)) \cap L^{2}(0,T;H^2(\Omega)),\;\;\sqrt{\tau} v\in L^{\infty}(0,T;H^{1}(\Omega))\\
 u_t\in L^{\frac{4}{3}}(0,T; (W^{1,4}(\Omega))^*),\;\;\tau v_t\in L^2(0,T;L^2(\Omega))\\
 \gamma(v)\nabla u,\;u\gamma'(v)\nabla v\in L^{\frac{4}{3}}(0,T;L^{\frac{4}{3}}(\Omega))
 \end{cases}
 \end{equation}such that
 \begin{equation}\label{1-7}
 	\langle u_t,\varphi \rangle+\int_\Omega \gamma(v)\nabla u\cdot\nabla\varphi dx+\int_\Omega u\gamma'(v)\nabla v\cdot\nabla \varphi dx=0,\qquad \forall \;\varphi\in W^{1,4}(\Omega), \;\; \text{a.e. in}\;(0,T)
 \end{equation}
 as well as
\begin{equation}
\tau v_{t}=\Delta v-v+u \quad \text{in} \quad  L^{2}((0,T)\times\Omega)\label{1-8}
\end{equation}
and 
\begin{equation}
	u(0,\cdot)=u_0\;\;\text{in}\;\; (W^{1,4}(\Omega))^*,\;\;\qquad \tau v(0,\cdot)=\tau v_0\;\;\text{in}\;\;L^2(\Omega).
\end{equation}
\end{theorem}
\begin{remark}\label{rem1}
The weak solution obtained above is a weak-strong solution  defined in \cite{Faguo}.
\end{remark}
\begin{remark}\label{rem2}
	Thanks to the upper bounds of $v$ derived by the comparison method, the regularity of our weak solutions is higher than those obtained in the literature \cite{Tao,Faguo,DKTY2019}. 
\end{remark}

Now let us to sketch the idea of our proof. Firstly, we would like to show that  with slight modification the comparison method originally proposed in \cite{GM,GM2} still works for our system under the assumption of \eqref{1-5} and $0\leq \tau<1/{K_\gamma}$. For a non-negative solution $(u,v)$ of system \eqref{1-3}, we introduce the following auxiliary elliptic Helmholtz problem:
\begin{equation}
\begin{cases}
-\Delta w+w=u, & x\in\Omega,\;  t>0 \\
\frac{\partial w}{\partial\nu}=0, &x\in \partial\Omega,\; t>0.\non
\end{cases}
\end{equation}
Then $w$ is well-defined and non-negative. Furthermore, we can derive from the first equation of system \eqref{1-3} the following key identity:
\begin{equation}\label{keyid}
	w_t+u\gamma(v)=(I-\Delta)^{-1}[u\gamma(v)].
\end{equation}
Here $\Delta$ denotes the usual Laplacian operator with homogeneous Neumann boundary condition.  Then with the upper bound assumption on $\gamma$, the non-negativity of $u\gamma(v)$ together with the comparison principle of elliptic equations, we can infer from above that
\begin{equation*}
w_t\leq (I-\Delta)^{-1}[u\gamma(v)]\leq K_\gamma (I-\Delta)^{-1}[u]=K_\gamma w,
\end{equation*}
which will give rise to a point-wise upper bound estimate that $w(x,t)\leq w_0(x)e^{tK_\gamma}$ by invoking Gronwall's inequality. Here $w_0=(I-\Delta)^{-1}[u_0]\geq0.$ Note that when $\tau=0$, $w$ is identical to $v$. Thus,  the upper bound of $v$ follows  if $\tau=0$.

On the other hand, if $\tau\in(0,1/{K_\gamma})$, we notice by the key identity that
\begin{equation*}
	\begin{split}
	\tau v_t-\Delta v+v=u&=\tau w_t-\Delta w+w-\tau w_t\\
	&=\tau w_t-\Delta w+w+\tau u\gamma(v)-\tau (I-\Delta)^{-1}[u\gamma(v)]\\
	&\leq \tau w_t-\Delta w+w+\tau K_\gamma u-\tau (I-\Delta)^{-1}[u\gamma(v)]
	\end{split}
\end{equation*}
Since $(I-\Delta)^{-1}[u\gamma(v)]\geq0$ by comparison principle of elliptic equations and $\tau K_\gamma<1$, we can deduce from above that
\begin{equation*}
	\tau v_t-\Delta v+v=u\leq \frac{1}{1-\tau K_\gamma}\left(\tau w_t-\Delta w+w\right).
\end{equation*}
Therefore, one can apply the comparison principle of heat equations to deduce that
\begin{equation*}
	v(x,t)\leq \frac{1}{1-\tau K_\gamma}\bigg(w(x,t)+K_0\bigg)\leq \frac{1}{1-\tau K_\gamma}\bigg(w_0(x)e^{tK_\gamma}+K_0\bigg)
\end{equation*}
where $K_0>0$ is a constant such that $v_0(x)\leq w_0(x)+K_0$ for all $x\in\overline{\Omega}$. As a result, the comparison method also works for the case $\tau\in(0,1/K_\gamma)$.

 One notices that with the specially chosen motility $\gamma(v)=e^{-v}$,  a solution $(u,v)$ satisfies that
 \begin{equation}\label{Lyapunov1}
 \frac{d}{dt}\mathcal{F}(u,v)(t)+\int_\Omega ue^{-v}\left|\nabla \log u-\nabla v\right|^2dx+\tau \|v_t\|^2_{L^2(\Omega)}=0,
 \end{equation}
 where \begin{equation*}
 \mathcal{F}(u,v)=\int_\Omega \left(u\log u+\frac12|\nabla v|^2+\frac12 v^2-uv\right)dx.
 \end{equation*} Such an energy-dissipation relation plays a key role in deriving adequate estimates to prove the weak solutions in \cite{Faguo}. 
 
  However, the system fails to possess such an entropy  with a generic $\gamma$ satisfying \eqref{1-5}. Thus, above mentioned comparison approach is used here to prove existence of weak solutions. Since the comparison method strongly relies on the structure of the system,  one needs to find a proper approximation procedure that conserves the delicate structure. With a $L^\infty$ cut-off type approximation scheme introduced in \cite{Tao} which is verified compatible with the comparison method, we prove the global existence of weak solutions by compactness argument in any dimension.
  
   Another difficulty in analysis lies the lack of regularity on the approximating solution $v_{\e t}$ when $\tau=0$ and hence the Aubin--Lions lemma cannot be applied directly. A trick used here is to derive  adequate uniform estimates for a family of approximating auxiliary functions $w_{\e}$ that satisfy a similar equation as the key identity \eqref{keyid}. Then one can prove the strong compactness of $w_{\e}$ by Aubin--Lions lemma. Furthermore,  the strong convergence of $v_\e$ follows by proving the difference  $w_\e-v_\e$ vanishes  as $\e\rightarrow0,$ where the uniform upper bounds of $w_\e$ and $v_\e$ derived by comparison method play a crucial role.

 Before concluding this section, we want to stress some new features of this work.
Firstly, we obtain the existence of the global weak solutions in any spatial dimension  with a generic motility satisfying \eqref{1-5}  for any $\tau\in[0,1/K_\gamma)$. Non-decreasing motility and sign-changing of  $\gamma'$ are also permitted in our case.  We remark that the comparison approach modified here can also be used to generalize the corresponding results in \cite{GM,GM2,FJ3}. Secondly, thanks to the upper bound estimates of $v$ given by the comparison method, our weak solution has higher regularity than those obtained in previous literature, see e.g., \cite{Faguo,Tao,DKTY2019}.

\section{Proof of Theorem \ref{1.1} }
In this section, we prove Theorem \ref{1-1} by compactness argument with the help of the modified comparison approach.
\subsection{Local existence of the approximate system}
Let us consider the following regularized problems:
\begin{equation}
\begin{cases}
u_{\varepsilon t}=\Delta ((\gamma(v_{\varepsilon})+\varepsilon)u_{\varepsilon}), & x\in\Omega,\; t>0  \\
\tau v_{\varepsilon t}=\Delta v_{\varepsilon}-v_{\varepsilon}+f_{\varepsilon}(u_{\varepsilon}), & x\in\Omega,\;  t>0 \\
\frac{\partial u_{\varepsilon}}{\partial \nu}=\frac{\partial v_{\varepsilon}}{\partial\nu}=0, &x\in \partial\Omega,\; t>0\\
u_{\varepsilon}(x,0)=u_{0\e}(x),\tau v_{\varepsilon}(x,0)=\tau v_{0\e}(x), &x\in\Omega \label{2-1}
\end{cases}
\end{equation}
where
\begin{equation}
f_{\varepsilon}(u_{\varepsilon})\triangleq\frac{u_{\varepsilon}}{1+\varepsilon u_{\varepsilon}}\quad \! \text{with} \quad \! \varepsilon  \in (0,\e_0)\label{2-2}
\end{equation}
with some fixed 
	$0<\e_0<\min\{1,\frac{1}{\tau}-K_{\gamma}\}.$ Here, $u_{0\e}$ and $v_{0\e}$ are non-negative regular functions such that $\|u_{0\e}\|_{L^\infty(\Omega)}\leq \|u_0\|_{L^\infty(\Omega)}$, $\|v_{0\e}\|_{L^\infty(\Omega)}\leq \|v_0\|_{L^\infty(\Omega)}$ and as $\e\rightarrow0$,  $u_{0\e}\rightarrow u_0$ in $L^2(\Omega)$ and  $v_{0\e}\rightarrow v_0$ in $H^1(\Omega)$.
\begin{lemma}\label{2.1}
	For each $\varepsilon \in (0,\e_0)$, there exists $T_{\mathrm{max},\e}\in(0,\infty]$ such that problem \eqref{2-1} possesses a non-negative classical solution $(u_{\varepsilon},v_{\varepsilon})$ in $\Omega\times(0,T_{\mathrm{max},\e})$.
	\begin{proof}
		For any fixed $\e,$ due to the fact $\varepsilon \leq \gamma(v_{\varepsilon})+\varepsilon\leq K_{\gamma}+\varepsilon$, one can obtain a pair  $(u_{\varepsilon},v_{\varepsilon})$ that solves \eqref{2-1} in the classical sense in $\Omega\times(0,T_{\mathrm{max},\e})$  by \cite[Lemma 2.1]{Tao}.
	\end{proof}
\end{lemma}
\subsection{Comparison method and the uniform upper bound of $v_{\varepsilon}$}
In this part, we established the uniform upper bound of $v_{\varepsilon}$ by the comparison method. Firstly, we introduce  auxiliary functions $w_{\varepsilon}(x,t)$ that satisfy the following equations:
\begin{equation}
\begin{cases}
w_{\varepsilon}-\Delta w_{\varepsilon}= u_{\varepsilon},& x\in\Omega ,\;t>0\\
\frac{\partial w_{\varepsilon}}{\partial \nu}=0 ,&  x\in \partial\Omega,\; t>0.\label{2-3}
\end{cases}
\end{equation}
Obviously, $w_\e$ well-defined on $\Omega\times(0,T_{\mathrm{max},\e})$ which is non-negative. Note for the approximation problem, $w_\e$ is no more identical to $v_\e$ if $\tau=0$.  Thus we need firstly establish a  point-wise upper bounds for $w_{\varepsilon}$ as follows.
\begin{lemma}\label{2.2} Assume $0\leq\tau< 1/K_\gamma$.
	For any $(x,t)\in\overline{\Omega}\times[0,T_{\mathrm{max},\e})$, there holds
	\begin{equation}\label{keyide}
		w_{\varepsilon t}+\big(\gamma(v_{\varepsilon})+\e\big)u_{\varepsilon}=(I-\Delta)^{-1}[(\gamma(v_{\varepsilon})+\varepsilon)u_{\varepsilon}].
	\end{equation}
	Moreover, for any $x\in\Omega$ and $t\in(0,T_{\mathrm{max},\e})$, we have
	\begin{equation*}
		w_{\varepsilon}(x,t)\leq w_{0\e}(x)e^{(K_{\gamma}+\e_0)t}
	\end{equation*}
	where $w_{0\e}\triangleq(I-\Delta)^{-1}[u_{0\e}]\in L^\infty(\Omega)$.
\end{lemma}
\begin{proof}
	First, a substitution of \eqref{2-3} into the first equation  of \eqref{2-1} yields that
	\begin{equation}
	-\Delta w_{\varepsilon t}+w_{\varepsilon t}=\Delta ((\gamma(v_{\varepsilon})+\varepsilon)u_{\varepsilon})\label{2-4}.
	\end{equation}
	Taking $(I-\Delta)^{-1}$ on both sides of the equality \eqref{2-4}, we obtain the  identity \eqref{keyide}.
	
	Due to fact $\gamma(v_{\varepsilon})+\varepsilon\leq K_{\gamma}+\e_0$, there holds $0\leq(\gamma(v_{\varepsilon})+\varepsilon)u_{\varepsilon}\leq (K_{\gamma}+\e_0)u_{\varepsilon}$ for any $(x,t)\in \overline{\Omega}\times(0,T_{\mathrm{max},\e})$. Then applying the comparison principle for elliptic equations, we deduce  that
	\begin{center}
		$0\leq(I-\Delta)^{-1}[(\gamma(v_{\varepsilon})+\varepsilon)u_{\varepsilon}]\leq (I-\Delta)^{-1}[(K_{\gamma}+\e_0)u_{\varepsilon}]=(K_{\gamma}+\e_0)w_{\varepsilon} $.
	\end{center}
	As a result, for any $(x,t)\in\overline{\Omega}\times[0,T_{\mathrm{max},\e})$, we obtain  by Gronwall's inequality that
	\begin{equation}
		w_{\varepsilon}(x,t)\leq w_{0}(x)e^{t(K_{\gamma}+\e_0)}.\non
	\end{equation} This completes the proof.
\end{proof}
Next, we derive the upper bounds of $v_\e$. First, we consider the case $\tau=0$.
\begin{lemma}\label{2.3}
	Assume $\tau=0$. Then for any $(x,t)\in\overline{\Omega}\times[0,T_{\mathrm{max},\e})$, there holds
	\begin{center}
		$v_{\varepsilon}(x,t)\leq w_{0}(x)e^{(K_{\gamma}+\e_0)t}$.
	\end{center}
	\begin{proof}
	First, we note that
		\begin{center}
			$ v_{\varepsilon}-\Delta v_{\varepsilon}=\frac{u_{\varepsilon}}{1+\varepsilon u_{\varepsilon}}\leq u_{\varepsilon}
			= w_{\varepsilon}-\Delta w_{\varepsilon}$.
		\end{center}
		Then applying the comparison principle for elliptic equations, we obtain by Lemma \ref{2-2} that
		\begin{equation*}
			 v_{\varepsilon}(x,t)\leq w_{\varepsilon}(x,t)\leq w_{0}(x)e^{(K_{\gamma}+\e_0)t}.
		\end{equation*}
	\end{proof}
\end{lemma}
Then we turn to the case $0<\tau<1/K_\gamma.$
\begin{lemma}\label{lm24}
	Assume $0<\tau<1/{K_{\gamma}}$. Then there is $K_0>0$ independent of $\e$ and time such that for any $(x,t)\in\overline{\Omega}\times[0,T_{\mathrm{max},\e})$
	\begin{equation}
		v_{\varepsilon}(x,t)\leq \frac{w_0(x)e^{(K_\gamma+\e_0)t}+K_0}{1-\tau(K_\gamma+\e_0)}.
	\end{equation}
	\begin{proof}
		Thanks to the second equation of \eqref{2-1} and the key identity \eqref{keyide}, we infer that
		\begin{align*}
		\tau v_{\varepsilon t}-\Delta v_{\varepsilon}+v_{\varepsilon}=&\frac{u_\e}{1+\e u_\e}\non\\
		\leq&u_\e=w_{\varepsilon}-\Delta w_{\varepsilon}\non\\
		=&w_{\varepsilon}-\Delta w_{\varepsilon}+\tau w_{\varepsilon t}-\tau(I-\Delta)^{-1}[(\gamma(v_{\varepsilon})+\varepsilon)u_{\varepsilon}]+\tau(\gamma(v_{\varepsilon})+\varepsilon)u_{\varepsilon}\non\\
		\leq & w_{\varepsilon}-\Delta w_{\varepsilon}+\tau w_{\varepsilon t}+\tau(K_\gamma+\varepsilon_0)u_{\varepsilon}
		\end{align*}due to the non-negativity of $\tau(I-\Delta)^{-1}[(\gamma(v_{\varepsilon})+\varepsilon)u_{\varepsilon}]$ and $u_\e$ together with  the fact $\e< \e_0$ and $\gamma\leq K_\gamma$.
		
Thus we  obtain  for any $(x,t)\in\Omega\times\overline{\Omega}\times[0,T_{\mathrm{max},\e})$ that
\begin{equation*}
	\begin{split}
	\tau v_{\varepsilon t}-\Delta v_{\varepsilon}+v_{\varepsilon}\leq u_\e\leq \frac{1}{1-\tau(K_\gamma+\e_0)}\bigg(\tau w_{\e t}-\Delta w_{\e}+w_\e\bigg).
	\end{split}
\end{equation*}	Now pick $K_0>0$ such that $v_0(x)\leq w_0(x)+K_0$ in $\overline{\Omega}$. Then	
by the comparison principle for heat equations, we deduce that
		\begin{equation*}
			v_{\varepsilon}(x,t)\leq \frac{w_\e(x,t)+K_0}{1-\tau(K_\gamma+\e_0)} 
		\end{equation*}which concludes the proof in view of Lemma \ref{2.2}.
	\end{proof}
\end{lemma}
Since $w_{0\e}\triangleq(I-\Delta)^{-1}[u_{0\e}]$, there holds $\|w_{0\e}\|_{L^\infty(\Omega)}\leq \|u_{0\e}\|_{L^\infty(\Omega)}\leq \|u_0\|_{L^\infty(\Omega)}$. Thus, $w_\e$ and $v_\e$ are both bounded from above by some $\e$-independent constant according to preceding lemmas, i.e., for any $0\leq \tau<1/K_\gamma,$ there is $v^*>0$ independent of $\e$ such that for all $(x,t)\in\overline{\Omega}\times[0,T_{\mathrm{max},\e})$
\begin{equation}\label{ubounds}
	w_\e(x,t),\;v_\e(x,t)\leq v^*.
\end{equation}
Now, we can extend the local classical solution $(u_\e,v_\e)$ globally.
\begin{lemma}
	For each $\varepsilon \in (0,\e_0)$, problem \eqref{2-1} possesses a non-negative classical solution $(u_{\varepsilon},v_{\varepsilon})$ on $\Omega\times(0,\infty)$. Moreover, we have the conservation of mass:
	\begin{equation*}
		\int_\Omega u_\e dx=\int_\Omega u_{0\e}dx\;\;\;\;\text{for all}\;t>0.
	\end{equation*}
\end{lemma}
\begin{proof}
Since  $0\leq v_\e(x,t)\leq v^*$  for all $(x,t)\in\overline{\Omega}\times[0,T_{\mathrm{max},\e})$, due to our assumption \eqref{1-5} on $\gamma$, there is a $k_{\gamma}, K_{\gamma'}>0$ which are independent of $\varepsilon$ such that $0<k_{\gamma}\leq\gamma(v_{\varepsilon})\leq K_{\gamma}$ as well as $|\gamma'|\leq K_{\gamma'}$ on $\overline{\Omega}\times[0,T_{\mathrm{max},\e})$.  Then we can argue in the same manner as in \cite[Lemma 5.1]{Tao} to prove that the classical solution obtained above indeed is a global one. The conservation of mass follows from a direct integration of the first equation in \eqref{2-1} over $\Omega$.
\end{proof}

\subsection{Uniform estimates}
In this part, we derive certain estimates for the global classical approximation solution $(u_\e,v_\e)$ that are independent of  $\varepsilon$.
\begin{lemma}\label{2.4}
Assume that $\tau\geq0$	and $(u_{\varepsilon},v_{\varepsilon})$ is a classical solution of system \eqref{2.3} on $\Omega \times (0,\infty)$. There exist $C>0$ depending on the $ u_{0},K_\gamma$ and $\Omega$ such that for all $\varepsilon \in  (0,\e_0), T>0$,
	\begin{equation*}
		\sup\limits_{0<t<T}\|u_{\varepsilon}(t)-\overline{u_{0\e}}\|^{2}_{H^{-1}(\Omega)}+\| w_{\varepsilon}(t)\|^{2}_{H^{1}(\Omega)}+\int^{T}_{0} \int_{\Omega}\gamma(v_{\varepsilon})u_{\varepsilon}^{2}dxdt\leq 2\| u_{0}-\overline{u_{0\e}}\|^{2}_{H^{-1}(\Omega)}+2\overline{u_{0\e}}^{2}|\Omega|+CT,
	\end{equation*}
	where $\overline{u_{0\e}}=\frac{1}{|\Omega|}\int_{\Omega}u_{0\e} dx$. In particular, there is $C(T)>0$ depending only on the initial data, $T,K_\gamma$ and $\Omega$ such that
	\begin{equation}
		\int_0^T\int_\Omega u^2_{\e }dxdt\leq C(T).
	\end{equation}
\end{lemma}
\begin{proof}
	By conservation of mass, one has $\overline{u_{\varepsilon}}=\overline{w_\e}=\overline{u_{0\e}}$. Multiplying the first equation by $(-\Delta)^{-1}(u_{\varepsilon}-\overline{u_{0\e}})$ and integrating over $\Omega$, we obtain that
	\begin{equation*}
		\frac{1}{2}\frac{d}{dt}\|(-\Delta)^{-\frac{1}{2}}(u_{\varepsilon}-\overline{u_{0\e}})\|^{2}_{L^{2}(\Omega)}+\int_{\Omega}(\gamma(v_{\varepsilon})+\varepsilon)u_{\varepsilon}^{2}dx=\overline{u_{0\e}}\int_{\Omega}(\gamma(v_{\varepsilon})+\varepsilon)u_{\varepsilon} dx.
	\end{equation*}
	Thanks to the fact that $\gamma(v_{\varepsilon})\leq K_{\gamma}$, we infer that
	\begin{equation*}
		\frac{1}{2}\frac{d}{dt}\|(-\Delta)^{-\frac{1}{2}}(u_{\varepsilon}-\overline{u_{0\e}})\|^{2}_{L^{2}(\Omega)}+\int_{\Omega}(\gamma(v_{\varepsilon})+\varepsilon)u_{\varepsilon}^{2}dx
		\leq (K_{\gamma}+1)\overline{u_{0\e}}^{2}|\Omega|,
	\end{equation*}
	which by a direct integration on $(0,T)$ with any $T\in(0,\infty)$ implies that
	\begin{equation*}
		\begin{split}
		&\sup\limits_{0\leq t\leq T}\|(-\Delta)^{-\frac{1}{2}}(u_{\varepsilon}(t)-\overline{u_{0\e}})\|^{2}_{L^{2}(\Omega)}+ 2 \int^{T}_{0}\int_{\Omega}(\gamma(v_{\varepsilon})+\varepsilon)u_{\varepsilon}^{2}dxdt \\
	\leq& \|(-\Delta)^{-\frac{1}{2}}(u_{0}-\overline{u_{0\e}})\|^{2}_{L^{2}(\Omega)}+2 (K_{\gamma}+1)\overline{u_{0\e}}^{2}|\Omega| T.
	\end{split}
	\end{equation*} Since $\overline{u_{0\e}}\leq\|u_0\|_{L^\infty(\Omega)}$ and $\gamma(v_\e)$ is bounded from above and below, there is $C(T)>0$ independent of $\e$ such that
	\begin{equation*}
		\int_0^T\int_\Omega u_{\e}^2dxdt\leq C_T.
	\end{equation*}
	On the other hand, we observe from  \eqref{2-3} that
	\begin{align*}
	\|w_{\varepsilon}\|_{H^1(\Omega)}^2=&\int_\Omega (|\nabla w_{\varepsilon}|^2+w_{\varepsilon}^2)dx\non\\
	=&\int_\Omega u_{\varepsilon}w_{\varepsilon}dx\non\\
	=&\int_\Omega (u_{\varepsilon}-\overline{u_{0\e}})w_{\varepsilon}dx+\overline{u_{0\e}}^2|\Omega|\non\\
	\leq& \|u_{\varepsilon}-\overline{u_{0\e}}\|_{H^{-1}(\Omega)}\|w_{\varepsilon}\|_{H^1(\Omega)}+\overline{u_{0\e}}^2|\Omega|.
	\end{align*}
	By Young's inequality,  we obtain that
	\begin{equation*}
		\| w_{\varepsilon} \|^{2}_{H^{1}(\Omega)} \leq\| u_{\varepsilon}-\overline{u_{0\e}}\|^{2}_{H^{-1}(\Omega)}+\overline{u_{0\e}}^{2}|\Omega|
	\end{equation*}which concludes the proof.
\end{proof}
\begin{lemma}\label{2.8}
	Assume $0\leq\tau<\frac{1}{K_{\gamma}}$. For any $T>0$, we can find $C(T)>0$ such that
	\begin{equation}\label{vh1}
	\tau\sup\limits_{0<t<T}\|v_\e(t)\|^2_{H^1(\Omega)}+\int_{0}^{T}\| v_{\varepsilon}\|_{H^2(\Omega)}^{2}\leq C(T)\quad for \quad  all \quad  \varepsilon \in (0,\e_0).
	\end{equation}\end{lemma}
\begin{proof} Note that $|f_\e(u_\e)|\leq u_\e$ a.e.
	When $0<\tau<\frac{1}{K_{\gamma}}$, we test the second equation in \eqref{2-1} by $-\Delta v_{\varepsilon}$ to obtain
	\begin{align*}
	\frac{\tau}{2}\frac{d}{dt}\int_{\Omega}|\nabla v_{\varepsilon}|^{2}+\int_{\Omega}|\Delta v_{\varepsilon}|^{2}+\int_{\Omega}|\nabla v_{\varepsilon}|^{2}=&-\int_{\Omega}f(u_{\varepsilon})\Delta v_{\varepsilon}\non\\
	\leq&\frac{1}{2}\int_{\Omega}|f(u_{\varepsilon})|^{2}+\frac{1}{2}\int_{\Omega}|\Delta v_{\varepsilon}|^{2}\non\\
	\leq&\frac{1}{2}\int_{\Omega}u_{\varepsilon}^{2}dx+\frac{1}{2}\int_{\Omega}|\Delta v_{\varepsilon}|^{2}.
	\end{align*}Then \eqref{vh1} follows from an integration with respect to time and  Lemma \ref{2.4}.
	If $\tau=0,$ the assertion follows directly from the second equation of \eqref{2-1} and Lemma \ref{2.4}.
\end{proof}

\begin{lemma}\label{2.5}
	Assume that $0\leq \tau <1/{K_{\gamma}}$ and $(u_{\varepsilon},v_{\varepsilon})$ is a classical solution of system \eqref{2-1} on $\Omega\times(0,\infty)$. For each $T>0$, we can find $C(T)>0$ depending on the initial data and $\Omega$ such that
	\begin{equation}
	\sup\limits_{0<t<T}\int_{\Omega} u_{\varepsilon}(t)\log u_{\varepsilon}(t) dx+\int_{0}^{T}\int_{\Omega}\gamma(v_{\varepsilon})\frac{| \nabla u_{\varepsilon}|^{2}}{u_{\varepsilon}}dxdt \leq C(T) \quad \text{for   all} \quad  \varepsilon \in  (0,\e_0)\non.
	\end{equation}
	\begin{proof}
		Multiplying the first equation of \eqref{2-1} by $\log u_{\varepsilon}$ and integrating over $\Omega$, we obtain that
		\begin{align}
		\frac{d}{dt}\int_\Omega u_{\varepsilon}\log u_{\varepsilon}dx+\int_\Omega (\gamma(v_{\varepsilon})+\varepsilon)\frac{|\nabla u_{\varepsilon}|^2}{u_{\varepsilon}}dx=-\int_\Omega \gamma'(v_{\varepsilon})\nabla v_{\varepsilon}\cdot \nabla u_{\varepsilon}dx\label{estlog}
		\end{align}
		where 
		\begin{align*}
		\left|\int_\Omega \gamma'(v_{\varepsilon})\nabla v_{\varepsilon}\cdot \nabla u_{\varepsilon}dx\right|
		\leq&\frac12\int_\Omega \gamma(v_{\varepsilon})\frac{|\nabla u_{\varepsilon}|^2}{u_{\varepsilon}}dx+\frac{1}{2}\int_\Omega
		\frac{|\gamma'(v_{\varepsilon})|^2}{\gamma(v_{\varepsilon})}u_{\varepsilon}|\nabla v_{\varepsilon}|^2dx\non\\
		\leq &\frac12\int_\Omega \gamma(v_{\varepsilon})\frac{|\nabla u_{\varepsilon}|^2}{u_{\varepsilon}}dx+\frac{1}{4}\int_\Omega \gamma(v_{\varepsilon})u_{\varepsilon}^2dx+
		\frac{1}{4}\int_\Omega\frac{|\gamma'(v_{\varepsilon})|^4}{\gamma(v_{\varepsilon})^3}|\nabla v_{\varepsilon}|^4dx.\non
		\end{align*}
		Invoking the Gagliardo--Nirenberg inequality, we obtain that
		\begin{equation}\nn
			\|\nabla v_{\varepsilon}\| _{L^{4}(\Omega)} \leq C\|v_{\varepsilon}\| _{H^{2}(\Omega)}^{\frac{1}{2}}\| v_{\varepsilon} \|_{L^{\infty}(\Omega)}^{\frac{1}{2}}+C\| v_{\varepsilon}\|_{L^{\infty}(\Omega)}.
		\end{equation}
		In view of Lemma \ref{2.8}, our assumption on $\gamma$ and the upper bound of $v$ given by Lemma \ref{2.3} \& Lemma \ref{lm24}, there is $C(T)>0$ depending on  $\gamma$ and the initial data such that
		\begin{align}
		\int_0^T\int_\Omega\frac{|\gamma'(v_{\varepsilon})|^4}{\gamma(v_{\varepsilon})^3}|\nabla v_{\varepsilon}|^4dx\leq& C(T)\int_0^T\int_\Omega|\nabla v_{\varepsilon}|^4dx\non\\
		\leq&C(T)\int_0^T\|v_{\varepsilon}\|^2_{H^2(\Omega)}+C(T)\leq C(T).\label{nv}
		\end{align}	
Thus we can conclude the proof by integrating \eqref{estlog} with respect to time.	
	\end{proof}
\end{lemma}
\begin{lemma}\label{2.6}
	Assume $0\leq \tau <\frac{1}{K_{\gamma}}$. Then for each $T>0$, one can find $C(T)>0$ depending on the initial data, $\gamma$ and $\Omega$ such that
	\begin{equation}
	\int_{0}^{T}\int_{\Omega}\big|\nabla u_{\varepsilon}\big|^{\frac{4}{3}}dxdt\leq C(T)\quad  \text{for    all} \quad   \varepsilon \in  (0,\e_0)\non.
	\end{equation}
	\begin{proof}
		Due to Lemma \ref{2.4} and Lemma \ref{2.5}, we  obtain by  Young's inequality that
		\begin{align}
		\int_{0}^{T}\int_{\Omega}\big|\nabla u_{\varepsilon}\big|^{\frac{4}{3}}dxdt=&\int_{0}^{T}\int_{\Omega}\big|\frac{\nabla u_{\varepsilon}}{\sqrt{u_{\varepsilon}}}\big|^{\frac{4}{3}}\big|\sqrt{u_{\varepsilon}}\big|^{\frac{4}{3}}dxdt \non\\
		\leq &C\int_{0}^{T}\int_{\Omega}\frac{|\nabla u_{\varepsilon}|^{2}}{u_{\varepsilon}}dxdt+ C\int_{0}^{T}\int _{\Omega}u_{\varepsilon}^{2}dxdt\leq C(T).\non
		\end{align}
	\end{proof}
\end{lemma}
\begin{lemma}\label{2.7}
	Assume $0\leq \tau <\frac{1}{K_{\gamma}}$. Then for any $T>0$, one can find $C(T) > 0$ such that
	\begin{equation}
	\int_{0}^{T}\|u_{\varepsilon t}(\cdot,t)\|^{\frac{4}{3}} _{(W^{1,4}(\Omega))^{*}} dt \leq C(T)\quad \text{for  all} \quad  \varepsilon \in (0,\e_0).\non
	\end{equation}
\end{lemma}
\begin{proof}
	Multiplying the first equation in \eqref{2-1} by an arbitrary $\psi \in L^{4}(0,T;W^{1,4}(\Omega))$ with $\|\psi\|_{L^4(0,T;W^{1,4}(\Omega))}=1$, integrating over $\Omega$ we see that
	\begin{align}
	\bigg|\int_0^T\int_{\Omega}u_{\varepsilon t}(\cdot,t)\psi\bigg|
	=&\bigg|\int_0^T\int_{\Omega}\nabla (u_{\varepsilon}(\gamma (v_{\varepsilon})+\varepsilon))\cdot\nabla \psi\bigg|\non\\
	=&\bigg|\int_0^T\int_{\Omega}\left(u_{\varepsilon}\gamma'(v_{\varepsilon})\nabla v_{\varepsilon}+(\gamma(v_{\varepsilon})+\varepsilon)\nabla u_{\varepsilon}\right)\cdot \nabla \psi\bigg|\non\\
	\leq& C\int_0^T\int _{\Omega}\left|u_{\varepsilon}\nabla v_{\varepsilon}\cdot\nabla \psi \right|+C\int_0^T\int_{\Omega}|\nabla u_{\varepsilon}\cdot\nabla \psi|,\non
	\end{align}
	where by Lemma \ref{2.6}
	\begin{equation*}
		\int_{0}^{T} \int _{\Omega}|\nabla u_{\varepsilon}\cdot\nabla \psi| dxdt\leq C \int_{0}^{T}\int_{\Omega}|\nabla u_{\varepsilon}|^{\frac{4}{3}}dxdt +C\int_{0}^{T}\int_{\Omega}|\nabla \psi |^{4}dxdt \leq C(T).
	\end{equation*}
	On the other hand, we can deduce by Lemma \ref{2.4} and \eqref{nv} that
	\begin{align}
	\int_{0}^{T}\int_{\Omega}\bigg|u_{\varepsilon}\nabla v_{\varepsilon}\cdot\nabla \psi \bigg|dxdt
	\leq & C \int_{0}^{T}\int_{\Omega}|\nabla v_{\varepsilon}|^{4}dxdt+C\int_{0}^{T}\int_{\Omega}|\nabla \psi |^{4}dxdt+C\int_{0}^{T}\int_{\Omega}u_{\varepsilon} ^{2}dxdt\non\\
	 \leq& C(T)\non.
	\end{align}
	As a consequence, 
	\begin{equation*}
	\begin{split}
		\| u_{\varepsilon t}(\cdot,t)\|_{L^{\frac43}(0,T;(W^{1,4}(\Omega))^{*})}
		\leq\left|\int_0^T\int_\Omega u_{\e t}\psi dxdt\right|
		\leq C(T).
	\end{split}
	\end{equation*}
\end{proof}

\begin{lemma}\label{2.9} Assume $0\leq \tau<1/K_\gamma.$
	For each $T>0$, one can find $C(T)>0$ such that
	\begin{equation}
	\int_{0}^{T}\int_{\Omega}w_{\varepsilon t}^{2}dxdt\leq C(T)\quad \text{for  all} \quad  \varepsilon \in (0,\e_0).\non
	\end{equation}
	\begin{proof}
	In view of the key identity \eqref{keyide} together with comparison principle of elliptic equations, we infer that
		\begin{center}
			$|w_{\varepsilon t}|^{2}\leq |(I-\Delta)^{-1}[(\gamma(v_{\varepsilon})+\varepsilon)u_{\varepsilon}
			]|^{2}+|(\gamma(v_{\varepsilon})+\varepsilon)u_{\varepsilon}|^{2}\leq (K_{\gamma}+1)^{2}(w_{\varepsilon}^{2}+u_{\varepsilon}^{2})$.
		\end{center}
	Then we obtain  by the Lemma \ref{2.2} and Lemma \ref{2.4} that
		\begin{equation*}
			\int_{0}^{T}\int_{\Omega}w_{\varepsilon t}^{2}dxdt\leq (K_{\gamma}+1)^{2}\left(\int_{0}^{T}\int_{\Omega}w_{\varepsilon}^{2}dxdt+\int_{0}^{T}\int_{\Omega}u_{\varepsilon}^{2}dxdt\right)\leq C(T).
		\end{equation*}
	\end{proof}
\end{lemma}
\subsection{Passage to the limit}
 Since  $(u_{\varepsilon},v_{\varepsilon})$ is a classical solution, for any given $\varphi \in C^{\infty}([0,T]\times \overline{\Omega})$, there holds
 \begin{equation}
 \int_{0}^{T}\int_{\Omega} u_{\varepsilon t}\varphi+\int_{0}^{T}\int_{\Omega}\gamma(v_{\varepsilon})\nabla u_{\varepsilon}\cdot \nabla \varphi+\int_{0}^{T}\int_{\Omega}u_{\varepsilon}\gamma^{'}(v_{\varepsilon})\nabla v_{\varepsilon}\cdot\nabla \varphi=0 \label{2-5}
 \end{equation}
 as well as
 \begin{equation}
 \tau\int_{0}^{T}\int_{\Omega} v_{\varepsilon t}\varphi+\int_{0}^{T}\int_{\Omega}\nabla v_{\varepsilon}\cdot \nabla \varphi-\int_{0}^{T}\int_{\Omega}v_{\varepsilon}\varphi=\int_{0}^{T}\int_{\Omega}f(u_{\varepsilon})\varphi \label{2-6}
 \end{equation}
for all $\varepsilon \in (0,\e_0)$.\\
Summarizing the $\e$-independent estimates obtained in previous part, we have
\begin{equation}
\{u_{\varepsilon}\}_{\varepsilon \in (0,\e_0)}\quad \text{ is bounded in} \quad  L^\infty(0,T;L^1\cap H^{-1}(\Omega))\cap L^{2}(0,T;L^{2}(\Omega))\cap L^{\frac{4}{3}}(0,T;W^{1,\frac{4}{3}}(\Omega)),\non
\end{equation}
\begin{equation}
\{v_{\varepsilon}\}_{\varepsilon \in (0,\e_0)}\quad  \text{ is bounded in} \quad L^{\infty}(0,T;L^{\infty}(\Omega))\cap L^{2}(0,T;H^2(\Omega)),\non
\end{equation}and by \eqref{nv} 
\begin{equation}
\{\nabla v_{\varepsilon}\}_{\varepsilon \in (0,\e_0)}\quad  \text{ is bounded in} \quad L^{4}(\Omega\times (0,T)).\non
\end{equation}
Recalling that by Lemma \ref{2.7}
\begin{equation}
\{u_{\varepsilon t}\}_{\varepsilon \in (0,\e_0)}\quad \text{ is bounded in} \quad L^{\frac{4}{3}}(0,T;(W^{1,4}(\Omega))^{*}),\non
\end{equation}
thanks to the Aubin--Lions Lemma, for any $T>0$, we get the existence of a subsequence (without relabeling)  such  that 
\begin{equation}
u_{\varepsilon}\rightarrow u  \quad  \text{in}\quad   L^{\frac{4}{3}} (\Omega \times (0,T))\quad \text{as} \quad \varepsilon\rightarrow 0\non,
\end{equation}
and hence
\begin{equation}
u_{\varepsilon}\rightarrow u \quad \text{a.e. }  \quad \Omega \times (0,T)\quad \text{as} \quad \varepsilon\rightarrow 0.\non
\end{equation}
Moreover, we can deduce that
\begin{equation}
u_{\varepsilon}\rightharpoonup u  \quad \text{in}  \quad L^{2}(\Omega\times(0,T))\quad \text{as}\quad  \varepsilon\rightarrow 0,\label{2-7}
\end{equation}
\begin{equation*}
u_{\varepsilon}\stackrel{*}{\rightharpoonup} u  \quad \text{in}  \quad L^{\infty}(0,T;L^1\cap H^{-1}(\Omega))\quad \text{as}\quad  \varepsilon\rightarrow 0,
\end{equation*}
\begin{equation}
\nabla u_{\varepsilon} \rightharpoonup \nabla u \quad \text{in} \quad  L^{\frac{4}{3}}(\Omega \times (0,T))\quad \text{as}\quad  \varepsilon\rightarrow 0,\label{2-8}
\end{equation}
and
\begin{equation}\label{weakut}
	u_{\varepsilon t}\rightharpoonup u_t  \quad \text{in}  \quad L^{\frac{4}{3}}(0,T;(W^{1,4}(\Omega))^*)\quad \text{as}\quad  \varepsilon\rightarrow 0
\end{equation}
Next, we aim to show the strong compactness of $v_\e$ in the case $\tau=0$. First, we show that $v_\e$ and $w_\e$ have the same limit.
\begin{lemma}\label{2.10}
	Assume $\tau=0$. For any $T>0$, there holds
	\begin{equation}
	\int_{0}^{T}\|w_{\varepsilon}-v_{\varepsilon}\|^{2}_{H^{1}(\Omega)}dt\rightarrow 0\quad \text{as} \quad  \varepsilon \rightarrow 0.\label{2-9}
	\end{equation}
	\begin{proof}
	A	subtraction of  \eqref{2-3} from the second equation of \eqref{2-1} yields that
		\begin{center}
			$(w_{\varepsilon}-v_{\varepsilon})-\Delta(w_{\varepsilon}-v_{\varepsilon})=u_{\varepsilon}-\frac{u_{\varepsilon}}{1+\varepsilon u_{\varepsilon}}$.
		\end{center}
		Multiplying above equality  by $(w_{\varepsilon}-v_{\varepsilon})$ and integrating the resultant over $\Omega$, we obtain
		\begin{align}
		\int_{\Omega}|w_{\varepsilon}-v_{\varepsilon}|^{2}dx+\int_{\Omega}|\nabla w_{\varepsilon}-\nabla v_{\varepsilon}|^{2}dx=&\int_{\Omega}(u_{\varepsilon}-\frac{u_{\varepsilon}}{1+\varepsilon u_{\varepsilon}})(w_{\varepsilon}-v_{\varepsilon})dx\non\\
		\leq &\|w_{\varepsilon}-v_{\varepsilon}\|_{L^{\infty}(\Omega)}\int_{\Omega} \left|u_{\varepsilon}-\frac{u_{\varepsilon}}{1+\varepsilon u_{\varepsilon}}\right|dx\non\\
		\leq& C(T)\int_{\Omega} \left|u_{\varepsilon}-\frac{u_{\varepsilon}}{1+\varepsilon u_{\varepsilon}}\right|dx.\non
		\end{align}
Since
		\begin{equation*}
			\left|u_{\varepsilon}-\frac{u_{\varepsilon}}{1+\varepsilon u_{\varepsilon}}\right|=\frac{\varepsilon u^{2}_{\varepsilon}}{1+\e u_{\varepsilon}}\leq \varepsilon u^{2}_{\varepsilon},
		\end{equation*}
	we infer by Lemma \ref{2.4} that as $\e\rightarrow0$,
		\begin{equation*}
		\int_0^T\|w_\e-v_\e\|^2_{H^1(\Omega)}\leq C(T)	\int_{0}^{T}\int_{\Omega}\left|u_{\varepsilon}-\frac{u_{\varepsilon}}{1+\varepsilon u_{\varepsilon}}\right|dxdt\leq C(T)\e\rightarrow 0 .
		\end{equation*}
	\end{proof}
\end{lemma}
Now, we may prove the strong compactness of $v_\e$ by showing the strong convergence of $w_\e$ and the latter can be proved due to the Aubin--Lions lemma.
\begin{lemma}\label{2.11}
Assume $\tau=0$. Then there is $v\in L^2(0,T;H^1(\Omega))$ such that
	\begin{equation}
	w_{\varepsilon},v_\e\rightarrow v \quad \text{in}\quad  L^{2}(0,T;H^{1}(\Omega)) \quad \text{as} \quad  \varepsilon \rightarrow 0\label{2-10}
	\end{equation}and
	\begin{equation}\label{wcvh2}
	v_\e\rightharpoonup v \quad \text{in}\quad  L^{2}(0,T;H^{2}(\Omega)) \quad \text{as} \quad  \varepsilon \rightarrow 0
	\end{equation}
\end{lemma}
\begin{proof}
	We observe from the elliptic regularity theorem and  Lemma \ref{2.4}  that
	\begin{equation*}
	\{w_{\varepsilon} \}_{\e\in(0,\e_0)}\quad\text{ is    bounded    in} \quad   L^{2}(0,T;H^{2}(\Omega)).
	\end{equation*}
	Due to the Lemma \ref{2.9}, we infer by the Aubin-Lions Lemma  that
	\begin{equation}
	w_{\varepsilon}\rightarrow v \quad \text{in}\quad  L^{2}(0,T;H^{1}(\Omega)) \quad \text{as} \quad  \varepsilon \rightarrow 0.\non
	\end{equation} Then convergence of $v_\e$ follows from Lemma \ref{2.10}.
\end{proof}

On the other hand, when $0<\tau<\frac{1}{K_{\gamma}}$, thanks to Lemma \ref{2.4} \& Lemma \ref{2.8}, we infer by the second equation of \eqref{2-1} that
\begin{equation}
\{v_{\varepsilon t}\}_{\e\in(0,\e_0)}\;\text{ is  bounded in }\; L^{2}(0,T;L^{2}(\Omega))\;\text{and} \; \{v_{\varepsilon}\}_{\e\in(0,\e_0)}\;\text{is bounded in}\; L^{2}(0,T;H^2(\Omega)).\non
\end{equation}Thus, we obtain that
\begin{equation}\label{weakvt}
	\tau v_{\varepsilon t}\rightharpoonup \tau v_{t} \;\;\text{in} \quad \! L^{2}(0,T;L^2(\Omega))\quad \text{as} \quad  \varepsilon \rightarrow 0.
\end{equation}
Applying the Aubin-Lions Lemma again, we have
\begin{equation}
v_{\varepsilon}\rightarrow v \quad \!\text{in} \quad \! L^{2}(0,T;H^{1}(\Omega))\quad \text{as} \quad  \varepsilon \rightarrow 0\label{2-12}
\end{equation}and hence $v_{\varepsilon}\rightarrow v$ a.e. in $\Omega\times(0,T).$ Moreover, since $v_\e$ has a uniform upper bound, we also have
\begin{equation*}
	v_{\varepsilon}\stackrel{*}{\rightharpoonup} v \quad \!\text{in} \quad \! L^{\infty}(0,T;L^\infty(\Omega))\quad \text{as} \quad  \varepsilon \rightarrow 0.
\end{equation*}
Now we are ready to discuss the convergence of nonlinear terms in \eqref{2-5} and \eqref{2-6}.
\begin{lemma}
	For any $T>0$, we have
	\begin{equation}
	\int_{0}^{T}\|\gamma'(v_{\varepsilon})\nabla v_{\varepsilon}-\gamma'(v)\nabla v\|^{2}_{L^{2}(\Omega)}dt\rightarrow 0 \quad \text{as} \quad \varepsilon \rightarrow 0\non
	\end{equation}and 
	\begin{equation*}
	\int_{0}^{T}\int_{\Omega} u_{\varepsilon}\gamma'(v_{\varepsilon})\nabla v_{\varepsilon}\cdot\nabla\varphi\rightarrow\int_{0}^{T}\int_{\Omega}u\gamma'(v)\nabla v\cdot \nabla\varphi \quad\text{ as} \quad   \varepsilon \rightarrow 0.
	\end{equation*}
	\begin{proof}
		First, by uniform boundedness of $|\gamma'(v_\e)|$, we note that
		\begin{align*}
		&\|\gamma'(v_{\varepsilon})\nabla v_{\varepsilon}-\gamma'(v)\nabla v\| _{L^{2}(\Omega\times(0,T))}\\
		\leq &\| \gamma'(v_{\varepsilon})(\nabla v_{\varepsilon}-\nabla v) \|_{L^{2}(\Omega\times (0,T))}  +  \|(\gamma'(v_{\varepsilon})-\gamma'(v))\nabla v\| _{L^{2}(\Omega\times(0,T))}\\
		\leq &C\| \nabla v_{\varepsilon}-\nabla v\| _{L^{2}(\Omega\times(0,T))} + \|(\gamma'(v_{\varepsilon})-\gamma'(v))\nabla v\| _{L^{2}(\Omega\times(0,T))}.
		\end{align*}
	Since  $v_{\varepsilon}\rightarrow v$ a.e. in $\Omega\times(0,T)$, we have
		\begin{equation}
		\gamma'(v_{\varepsilon})\rightarrow \gamma'(v) \quad  \text{a.e. in} \quad \Omega\times (0,T).
		\end{equation}
		By the dominated convergence theorem, we  get
		\begin{align*}
		&\| (\gamma'(v_{\varepsilon})-\gamma'(v))\nabla v\|_{L^{2}(\Omega\times(0,T))}^{2}\\
		=&\int_{0}^{T}\int_{\Omega}| \gamma'(v_{\varepsilon})-\gamma'(v)|^{2}|\nabla v|^{2} \rightarrow 0 \quad \text{as} \quad  \varepsilon \rightarrow 0.
		\end{align*}
Thus the first convergence follows in view of \eqref{2-12}. Then we may complete the proof thanks to \eqref{2-7}.
	\end{proof}
\end{lemma}

\begin{lemma}
For any $T>0$, we have
\begin{equation}
\int_{0}^{T}\int_{\Omega}\gamma(v_{\varepsilon})\nabla u_{\varepsilon}\cdot\nabla\varphi\rightarrow \int _{0}^{T}\int_{\Omega}\gamma(v)\nabla u\cdot\nabla \varphi  \quad \text{as} \quad \varepsilon \rightarrow 0\non
\end{equation} and
	\begin{equation}
\int_{0}^{T}\int_{\Omega}\frac{u_{\varepsilon}}{1+\varepsilon u_{\varepsilon}}\varphi\rightarrow\int_{0}^{T}\int_{\Omega}u\varphi \quad \text{as} \quad   \varepsilon\rightarrow 0.\non
\end{equation}
\end{lemma}
\begin{proof}
First, we note that $\gamma(v_\e)\rightarrow\gamma(v)$ a.e. in $\Omega\times(0,T)$. Moreover, the uniform boundedness of $\gamma(v_\e)$ together with the dominated convergence theorem entails that 
\begin{equation*}
	\|\gamma(v_\e)\|_{L^p(\Omega\times(0,T))}\rightarrow 	\|\gamma(v)\|_{L^p(\Omega\times(0,T))},\;\;\forall\;1<p<\infty\qquad\text{as}\;\;\e\rightarrow0.
\end{equation*}Thus, we obtain that
\begin{equation*}
\gamma(v_{\varepsilon})\rightarrow \gamma(v) \quad \text{in}  \quad L^{p}(\Omega\times(0,T)),\;\;\forall\;1<p<\infty \quad \text{as} \quad   \varepsilon \rightarrow 0,
\end{equation*}
which together with \eqref{2-8} implies that
\begin{equation}
\int_{0}^{T}\int_{\Omega}\gamma(v_{\varepsilon})\nabla u_{\varepsilon}\cdot\nabla\varphi\rightarrow \int _{0}^{T}\int_{\Omega}\gamma(v)\nabla u\cdot\nabla \varphi  \quad \text{ as} \quad \varepsilon\rightarrow 0.\non
\end{equation}

On the other hand, since \begin{align}
&\int_{0}^{T}\int_{\Omega}\left|\frac{1}{1+\varepsilon u_{\varepsilon}}-1\right|^{2}dxdt
=\int_{0}^{T}\int_{\Omega}\left|\frac{\varepsilon u_{\varepsilon}}{1+\varepsilon u_{\varepsilon}}\right|^{2}dxdt
\leq\varepsilon ^{2}\int_{0}^{T}\int_{\Omega} u_{\varepsilon}^{2}dxdt\non
\end{align}we infer by Lemma \ref{2.4} that
\begin{equation*}
\frac{1}{1+\varepsilon u_{\varepsilon}}\rightarrow 1 \quad  \text{in} \quad  L^{2}(\Omega\times(0,T)) \quad\text{ as} \quad  \varepsilon\rightarrow 0,
\end{equation*}
and thus by \eqref{2-7}
\begin{equation*}
\frac{u_{\varepsilon}}{1+\varepsilon u_{\varepsilon}}\rightharpoonup u \quad  \text{in} \quad  L^{1}(\Omega\times(0,T)) \quad \text{as} \quad  \varepsilon\rightarrow 0.\label{2-17}
\end{equation*}This completes the proof.
\end{proof}

\noindent\textbf {Proof of Theorem \ref{1.1}. }

 Now, we are ready to pass to the limit in \eqref{2-5} and \eqref{2-6} to check that the limit function $(u,v)$ satisfies  \eqref{2-5} and \eqref{2-6} by replacing $(u_\e,v_\e)$. Besides, in view of the obtained uniform estimates obtained above,  $(u,v)$ also satisfies the regularity \eqref{regularity} stated in Theorem \ref{1.1}. Thus $(u,v)$ fulfills \eqref{1-7} and \eqref{1-8} since $C^\infty([0,T]\times\overline{\Omega})$ is dense in $L^4(0,T;W^{1,4}(\Omega))$ as well as $L^2(0,T;L^2(\Omega))$. The initial data can be justified based on the weak convergences \eqref{2-7}, \eqref{weakut}, \eqref{wcvh2}, \eqref{weakvt} together with the uniqueness of limit. This completes the proof.\qed

\end{document}